\documentclass[12pt]{amsart}
\usepackage{amsfonts}
\usepackage{amsmath}
\usepackage{amssymb}
\usepackage{graphicx}
\usepackage{amscd}
\usepackage{xcolor}
\usepackage{hyperref}

\newtheorem{theorem}{Theorem}

\newtheorem{lemma}[theorem]{Lemma}

\newtheorem{proposition}[theorem]{Proposition}

\begin{document}

\begin{abstract}
Let $X$ be a non-compact symmetric space of dimension $n$. We prove that if $f\in L^{p}(X)$, $1\leq p\leq 2$, then the Riesz means $S_{R}^{z}\left(
f\right) $ converge to $f$ almost
everywhere as $R\rightarrow \infty $, whenever $\operatorname{Re}z>\left( n-\frac{1}{2}\right) \left( \frac{2}{p}-1\right) $.
\end{abstract}
\dedicatory{To the memory of Professor Michel Marias.}
\title[Riesz means]{Riesz means on symmetric spaces}
\author{A. Fotiadis and  E. Papageorgiou }
\title{Riesz means on symmetric spaces}
\thanks{The second author is supported by the Hellenic Foundation for Research and Innovation, Project HFRI-FM17-1733.}
\curraddr{Department of Mathematics, Aristotle University of Thessaloniki,
Thessaloniki 54.124, Greece }
\maketitle

\section{Introduction and statement of the results}

In this article we study the almost everywhere convergence of the Riesz means on a noncompact symmetric space of arbitrary rank. To state our results, we need to introduce some notation. 

Let $G$ be a
semi-simple, noncompact, connected Lie group with finite center and let $K$
be a maximal compact subgroup of $G$. We consider the $n$-dimensional symmetric space of
noncompact type $X=G/K$, and let $\operatorname{dim}%
X=n $. Denote by $\mathfrak{g}$ and $\mathfrak{k}$ the Lie
algebras of $G$ and $K$, respectively. We have the Cartan decomposition $%
\mathfrak{g}=\mathfrak{p}\oplus \mathfrak{k}$. Let $\mathfrak{a}$ be a
maximal abelian subspace of $\mathfrak{p}$ and $\mathfrak{a}^{\ast }$ its
dual. If $\operatorname{dim}\mathfrak{a}=l$, we say that $X$ has rank $l$.

The Killing form on $\mathfrak{g}$ restricts to a positive definite form on $\mathfrak{a}$, which in turn induces a positive inner product and hence a norm $\| \cdot \|$ on $\mathfrak{\ a^{\ast}}$. Denote by $\rho $ the half sum of positive roots, counted with their
multiplicities. Fix $R\geq \Vert \rho \Vert ^{2}$ and $z\in \mathbb{C}$ with
$\operatorname{Re}z\geq 0$, and consider the bounded function
\begin{equation}
s_{R}^{z}(\lambda )=\left( 1-\frac{\Vert \rho \Vert ^{2}+\Vert \lambda \Vert
^{2}}{R}\right) _{+}^{z},\;\lambda \in \mathfrak{a^{\ast }}.  \label{mult}
\end{equation}

Denote by $\kappa _{R}^{z}$ its inverse spherical Fourier transform in the
sense of distributions and consider the so-called Riesz means operator $%
S_{R}^{z}$:
\begin{equation}
S_{R}^{z}(f)(x)=\int_{G}f(y)\kappa_{R}^{z}(y^{-1}x)dy=(f\ast\kappa _{R}^{z})(x),\quad f\in C_{0}(X).  \label{operatorX}
\end{equation}

For every pair $p,q$ such that $1\leq p,q\leq \infty $, denote by $%
(L^{p}+L^{q})(X)$ the Banach space of all functions $f$ on $X$ which admit a
decomposition $f=g+h$ with $g\in L^{p}$ and $h\in L^{q}$. The norm of $f\in
(L^{p}+L^{q})(X)$ is given by
\begin{equation}\label{norm}
\Vert f\Vert _{\left( p,q\right) }=\inf \left\{ \Vert f\Vert _{p}+\Vert
g\Vert _{q}:\text{ for all decompositions }f=g+h\right\} .
\end{equation}%
For $q\geq 1$, denote by $q^{\prime }$ its conjugate. In the present work we
prove the following results.

\begin{theorem}
\label{th0}Let $z\in \mathbb{C}$ with $\operatorname{Re}z\geq n-\frac{1}{2}$ and
consider $q>2$. Then, for every $p$ such that $1\leq p\leq q^{\prime }$, and
for every $r\in \lbrack qp^{\prime }/(p^{\prime }-q),\infty ]$, $S_{R}^{z}$
is uniformly bounded from $L^{p}(X)$ to $(L^{p}+L^{r})(X)$.
\end{theorem}

Next we deal with the maximal operator $S_{\ast }^{z}$ associated with Riesz
means:
\begin{equation*}
S_{\ast }^{z}(f)(x)=\sup_{R>\Vert \rho \Vert ^{2}}|S_{R}^{z}(f)(x)|,\;f\in
L^{p}({X}),\;1\leq p\leq 2.
\end{equation*}%
Set
\begin{equation*}
Z_0(n,p)=\left( n-\frac{1}{2}\right) \left( \frac{2}{p}-1\right) .
\end{equation*}%
We have the following result.

\begin{theorem}
\label{Th1} Let $1\leq p\leq 2$ and consider $q>2$. If $\operatorname{Re}z>Z_0(n,p)$,
then for every $s\geq pq/(2-p+pq-q)$, there is a constant $c\left( z\right) >0$, such
that for every $f\in L^{p}(X)$,
\begin{equation*}
\Vert S_{\ast }^{z}f\Vert _{\left( p,s\right) }\leq c(z)\Vert f\Vert _{p}.
\end{equation*}
\end{theorem}

Note that the $(p,s)$ norm is defined in (\ref{norm}). As a corollary of the Theorem \ref{Th1}, we obtain the almost everywhere
convergence of Riesz means.

\begin{theorem}
\label{Th2} Let $1\leq p\leq 2$. If $\operatorname{Re}z>Z_0(n,p)$, then for $f\in
L^{p}(X)$,
\begin{equation}
\lim\limits_{R\rightarrow +\infty }S_{R}^{z}f(x)=f(x),\;\text{a.e.. }
\label{aec}
\end{equation}
\end{theorem}

Our result treats the general case of noncompact symmetric spaces of \textit{all} ranks. It is interesting that the index $Z_0(n,p)$ only depends on the Euclidean dimension of $X$ and not on the rank of $X$. The only known results studying the Riesz means on noncompact symmetric spaces are \cite{GIUMA, zhu}, where the authors treat the case of rank one noncompact symmetric spaces, as well as the case of arbitrary rank when $G$ is complex, and the case of $SL(3,\mathbb{H})/Sp(3)$ respectively. 

Here we treat the general case of noncompact symmetric spaces
of all ranks, by using the inverse Abel transform. This way we can study the general case of a noncompact symmetric space, an area that remained inactive since the seminal work \cite{GIUMA} in 1991. The price we pay is that our result is valid for $\operatorname{Re}z$ larger than $Z_0(n,p)=\left( n-\frac{1}{2}\right) \left( \frac{2}{p}-1\right) $. Note that in the setting of $\mathbb{R}^{n}$, \cite{STEIN27}, as well as in
case of the rank one symmetric spaces, \cite{GIUMA}, (\ref{aec}) is valid
for $\operatorname{Re}z$ larger than the critical index $z_0(n,p)=\left( \frac{n-1}{2}\right) \left( \frac{2}{p}-1\right) $. Thus, we can treat the arbitrary rank case but our result is not optimal, as a consequence of the lack of an explicit formula for the inverse Abel transform in the general case of a symmetric space.

Many authors have investigated the almost everywhere convergence
of Riesz means. They have already been extensively studied in the case of $%
\mathbb{R}^{n}$ (\cite{CHRIST7,CHRIST8,STEIN27,STEIN29} as well as in the
book \cite{DAVCHANG}). In the case of elliptic differential operators on
compact manifolds they are treated in (\cite%
{Berard,ChristSogge,GIUTRA,HORM,SEEGER,Sogge}). The case of Lie groups of
polynomial volume growth and of Riemannian manifolds of nonnegative curvature
is studied in \cite{ALEXOLO,MAR} and the case of compact semisimple Lie
groups in \cite{CLERC2}.

To prove Theorem \ref{th0}, we split the Riesz means operator in the sum of
two convolution operators: $S_{R}^{z}=S_{R}^{z,0}+S_{R}^{z,\infty }$. The
\textit{local part} $S_{R}^{z,0}$ has a compactly supported kernel around
the origin, while the kernel of the \textit{part at infinity} $%
S_{R}^{z,\infty }$ is supported away from the origin. To treat the local
part, we follow the approach of \cite{ALEXOLO, PAPAGEO}. More precisely, we
express the kernel of $S_{R}^{z,0}$ via the heat kernel $p_{t}$ of $X$, and
we make use of its estimates. Let $-\Delta$ be the Laplace-Beltrami operator on $X$. Then, combining the with the fact that the wave
operator $\cos(t\sqrt{-\Delta-\|\rho\|^2})$ of $X$ propagates with finite speed, allows
us to prove that $S_{R}^{z,0}$ is continuous on $L^{p}(X)$ for all $p\geq 1$%
. To treat the part at infinity of the operator, we proceed as in \cite%
{LOMAanna}, and obtain estimates of its kernel by using the support
preserving property of the \textit{Abel transform}. 

This paper is organized as follows. In Section 2 we present the necessary
ingredients for our proofs. In Section 3 we deal with the local part
and the part at infinity, of the Riesz mean operator and we
prove Theorem \ref{th0}. In Section 4 we prove Theorem \ref{Th1} and we
deduce Theorem \ref{Th2}.

\section{Preliminaries}

In this section we recall some basic facts about symmetric spaces. For
details see for example \cite{AN, FOMAMA, HEL, LOMAjga}.

\subsection{Symmetric spaces}

Let $G$ be a semisimple Lie group, connected, noncompact, with finite center
and let $K$ be a maximal compact subgroup of $G$. We denote by $X$ the
noncompact symmetric space $G/K$. In the sequel we assume that $\operatorname{dim}%
X=n $. Denote by $\mathfrak{g}$ and $\mathfrak{k}$ the Lie algebras of $G$
and $K $. Let also $\mathfrak{p}$ be the subspace of $\mathfrak{g}$ which is
orthogonal to $\mathfrak{k}$ with respect to the Killing form. The Killing
form induces a $K$-invariant scalar product on $\mathfrak{p}$ and hence a $G$%
-invariant metric on $X$. Denote by $\Delta $ the Laplace-Beltrami operator
on $X$, by $d(.,.)$ the Riemannian distance and by $dx$ the associated
Riemannian measure on $X$. Denote by $\left\vert B\left( x,r\right)
\right\vert $ the volume of the ball $B\left( x,r\right) $, $x\in X$, $r>0$,
and recall that there is a $c>0$, such that  \begin{equation}\label{vol}\left\vert B\left( x,r\right)
\right\vert \leq cr^{n} \text{ for all } r\leq 1,\end{equation} \cite[p.117]{stromberg}.

Fix $\mathfrak{a} $ a maximal abelian subspace of $\mathfrak{p}$ and denote
by $\mathfrak{a}^{\ast }$ the real dual of $\mathfrak{a}$. If $\operatorname{dim}%
\mathfrak{a}=l$, we say that $X$ has rank $l$. We also say that $\alpha \in
\mathfrak{a}^{\ast }$ is a root vector, if the space
\begin{equation*}
\mathfrak{g}^{\alpha }=\left\{ X\in \mathfrak{g}:[H,X]=\alpha (H)X,\text{
for all }H\in \mathfrak{a}\right\} \neq \left\{ 0\right\} .
\end{equation*}

Let $A$ be the analytic subgroup of $G$ with Lie algebra $\mathfrak{a}.$ Let
$\mathfrak{a}_{+}\subset \mathfrak{a}$ be a positive Weyl chamber and let $%
\overline{\mathfrak{a}_{+}}$ be its closure. Set $A^{+}=\exp \mathfrak{a}%
_{+} $. Its closure in $G$ is $\overline{A_{+}}=\exp \overline{\mathfrak{a}
_{+}}$. We have the Cartan decomposition
\begin{equation}
G=K\overline{A_{+}}K=K\exp \overline{\mathfrak{a}_{+}}K.  \label{kak}
\end{equation}
Then, each element $x \in G $ is written uniquely as $x =k_1(\exp H)k_2$. We
set
\begin{equation}  \label{normG}
|x|=|H|, \; H \in \overline{\mathfrak{a}_{+}},
\end{equation}
the norm on $G$ \cite[p.2]{ANO}. Denote by $x_0 = eK$ a base point of $X$.
If $x,y\in X$, there are isometries $g,h \in G $ such that $x = gx_0$ and $y
= hx_0$. Because of the Cartan decomposition (\ref{kak}), there are $k,
k^{\prime}\in K$ and a unique $H \in \overline{\mathfrak{a}_{+}}$ with $%
g^{-1}h=k \exp H k^{\prime}$ . It follows that
\begin{equation*}
\ d(x, y)=|H|,
\end{equation*}
where $d(x,y)$ is the geodesic distance on $X$ \cite{Weber}.

Normalize the Haar measure $dk$ of $K$ such that $\int_{K}dk=1$. Then, from
the Cartan decomposition, it follows that
\begin{equation}  \label{haarmeas}
\int_{G}f(g)dg=\int_{K}dk_{1}\int_{{\mathfrak{a}_{+}}}{\delta (H)}%
dH\int_{K}f(k_{1}\exp (H)k_{2})dk_{2},
\end{equation}
where the modular function $\delta (H)$ satisfies the estimate
\begin{equation}  \label{modular}
\delta(H)\leq ce^{2\rho (H)} .
\end{equation}
We identify functions on $X = G/K$ with functions on $G$ which are $K$%
-invariant on the right, and hence bi-$K$-invariant functions on $G$ are
identified with functions on $X$, which are $K$-invariant on the left. Note
that if $f$ is $K$-bi-invariant, then by (\ref{haarmeas}),
\begin{equation}
\int_{G}f\left( g\right) dg=\int_{X}f\left( x\right) dx=\int_{\mathfrak{a}%
_+}f(\exp H)\delta (H)dH.  \label{rightinv}
\end{equation}

\subsection{The spherical Fourier transform}

Denote by $S(K\backslash G/K)$ the Schwartz space of $K$-bi-invariant
functions on $G$. For $f\in S(K\backslash G/K),$ the spherical Fourier
transform $\mathcal{H}$ is defined by
\begin{equation*}
\ \mathcal{H}f(\lambda )=\int_{G}f(x)\varphi _{\lambda }(x)\;dx,\quad
\lambda \in \mathfrak{a^{\ast }},
\end{equation*}
where $\varphi _{\lambda }$ is the elementary spherical function of index $%
\lambda$ on $G$. Note that from \cite{HEL} we have the following estimate
\begin{equation}  \label{phi0}
\varphi_0(\exp H) \leq c(1+|H|)^de^{-\rho(H)},
\end{equation}
where $d$ is the cardinality of the set of positive indivisible roots.

Let $S(\mathfrak{a^{\ast }})$ be the usual Schwartz space on $\mathfrak{\
a^{\ast}}$. Denote by $W$ the Weyl group associated to the root system of $(%
\mathfrak{g},\mathfrak{a})$ and denote by $S(\mathfrak{a^{\ast }})^{W}$ the
subspace of $W$-invariant functions in $S(\mathfrak{a^{\ast }})$. Then, by a
celebrated theorem of Harish-Chandra, $\mathcal{H}$ is an isomorphism
between $S(K\backslash G/K)$ and $S(\mathfrak{a^{\ast }})^W$ and its inverse
is given by
\begin{equation*}
\ (\mathcal{H}^{-1}f)(x)=c\int_{\mathfrak{a\ast }}f(\lambda )\varphi
_{-\lambda }(x)\frac{d\lambda }{|\mathbf{c}(\lambda )|^{2}},\quad x\in
G,\quad f\in S(\mathfrak{a^{\ast }})^{W},
\end{equation*}
where $\mathbf{c}(\lambda )$ is the Harish-Chandra function and $c$ is a normalizing constant independent of $f$, \cite[Theorem 7.5]{HEL}.

\subsection{The heat kernel on $X$}

 Set
\begin{equation*}
w_t(\lambda )=e^{-t(\Vert \lambda \Vert ^{2}+\Vert \rho \Vert
	^{2})},\quad t>0,\;\lambda \in \mathfrak{a^{\ast }},
\end{equation*} Then the heat kernel $p_{t}(x)$ of $X$ is given by $(\mathcal{H}^{-1}w_t)(x)$ \cite{ANJ}.

The heat kernel $p_t$ on symmetric spaces has been extensively studied, see for example \cite{ANJ,ANO}. Sharp estimates of the heat kernel have been obtained by Davies and Mandouvalos in \cite{DAVMAN} for the case of real hyperbolic space,
while Anker and Ji \cite{ANJ} and later Anker and Ostellari \cite{ANO},
generalized the results of \cite{DAVMAN} to all symmetric spaces of
noncompact type.

Denote by $\Sigma _{0}^{+}$ the set of positive indivisible roots $\alpha $ of $(\mathfrak{g}, \mathfrak{a})$
and by $m_{\alpha }$ the dimension of the root space $\mathfrak{g}^{\alpha }$. In \cite[Main Theorem]{ANO} it is proved the following sharp estimate:
\begin{align}
p_{t}(\exp {H})& \leq ct^{-n/2}\left( \underset{\alpha \in \Sigma _{0}^{+}}{	\prod }(1+\langle \alpha ,H\rangle )(1+t+\langle \alpha ,H\rangle )^{\frac{		m_{\alpha }+m_{2\alpha }}{2}-1}\right) \times \notag \\
& \times e^{-\left\Vert \rho \right\Vert ^{2}t-\langle \rho ,H\rangle
	-|H|^{2}/4t},\quad t>0, \;H\in \overline{\mathfrak{a}	_{+}},  \label{heat}
\end{align}where $n=$dim$X$.

From (\ref{heat}), we deduce the following crude estimate
\begin{equation}
p_{t}(\exp {H})\leq ct^{-n/2}e^{-|H| ^{2}/4t}, \quad
t>0, \;H\in \overline{\mathfrak{a}_{+}},  \label{ostellari}
\end{equation}which is sufficient for our purposes.

Note also that (\ref{ostellari}) yields the on-diagonal upper bound 
	\begin{equation}\label{diagonal} p_t(e)\leq ct^{-n/2}.
	\end{equation}   As it is shown in \cite[Lemma 3.1]{GRIG}, 
estimate (\ref{diagonal}) implies that there is an absolute constant  $D>0$, sufficiently large, such that for every $a>0$, there holds
\begin{equation}
\int_{d(x,x_{0})>a}p_{t}^{2}(x)dx\leq ct^{-n/2}e^{-a^{2}/Dt}.
\label{grigoryan}
\end{equation}
\section{Proof of Theorem 1}
	Let $\kappa _{R}^{z}$ be the kernel of the Riesz means operator. We start
	with a decomposition of $\kappa _{R}^{z}$:
	\begin{equation}
	\kappa _{R}^{z}=\zeta \kappa _{R}^{z}+(1-\zeta )\kappa _{R}^{z}:=\kappa
	_{R}^{z,0}+\kappa _{R}^{z,\infty },  \label{decomp}
	\end{equation}%
	where $\zeta \in C^{\infty }(K\backslash G/K)$ is a cut-off function such
	that
	\begin{equation}
	\zeta (x)=%
	\begin{cases}
	1, & \text{if }|x|\leq 1/2, \\
	0, & \text{if }|x|\geq 1.%
	\end{cases}
	\label{zita}
	\end{equation}%
	Denote by ${S}_{R}^{z,0}$ (resp. ${S}_{R}^{z,\infty }$) the convolution
	operator on $X$ with kernel $\kappa _{R}^{z,0}$ (resp. $\kappa
	_{R}^{z,\infty })$. 
\subsection{The local part}

We shall prove the following proposition.

\begin{proposition}\label{T0X} Assume that $\operatorname{Re}z >n/2$. Then the operator ${S}_{R}^{z,0}$
is bounded on $L^{p}(X)$, $1\leq p\leq \infty$, and $\|S_R^{z,
0}\|_{p\rightarrow p}\leq c(z)$, for some constant $c(z)>0$.
\end{proposition}

The proof is lengthy and it will be given in several steps. First, we
shall express the kernel $\kappa _{R}^{z}$ in terms of the heat kernel $%
p_{t} $ of $X$. Then, we shall use the heat kernel estimates (\ref{ostellari}%
) to prove that $\kappa _{R}^{z}$ is integrable in the unit ball $B(0,1)$ of
$X$. This implies that $S_{R}^{z,0}$ is bounded on $L^{\infty }(X)$. We then prove that $S_{R}^{z,0}$ is bounded on $L^2(X)$, and an
interpolation argument between $L^{\infty}(X)$ and $L^2(X)$
allows us to conclude.

To express the kernel $\kappa _{R}^{z}$ in terms of $p_{t}$, we follow \cite%
{ALEXOLO} and we write
\begin{equation}\label{Srz}
s_{R}^{z}(\lambda )=s_{R}^{z}(\Vert \lambda \Vert ) =\left( 1-\frac{\Vert
\lambda \Vert ^{2}+\Vert \rho \Vert ^{2}}{R}\right) _{+}^{z}.  
\end{equation}  
Set $r=\sqrt{R}$, $\xi =\Vert \lambda \Vert $ and consider the function
\begin{equation}
h_{r}^{z}(\lambda )=h_{r}^{z}(\xi ):=\left( 1-\left( \frac{\sqrt{\xi ^{2}+\Vert \rho
\Vert ^{2}}}{r} \right)^2\right) _{+}^{z}e^{(\sqrt{\xi ^{2}+\Vert \rho \Vert ^{2}}/{%
r})^{2}}.  \label{hr,z}
\end{equation}%

Then, from (\ref{Srz}) and (\ref{hr,z})
we have
\begin{equation}
s_r^z(\lambda)=h_R^z(\lambda)e^{-(\|\lambda\|^2+\| \rho \|^{2})/r^2},
\end{equation}
and thus
\begin{equation}
s_{R}^{z}(\sqrt{-\Delta-\|\rho\|^2} )=h_{r}^{z}(\sqrt{-\Delta-\|\rho\|^2})e^{-1/{r^{2}}(-\Delta )}. 
\label{heatexpr}
\end{equation} 

Next, we recall the construction of the partition of unity of %
\cite[p.213]{ALEXOLO} we shall use for the splitting of the
operator $s_{R}^{z}(-\Delta )$. For that we set $\psi (\xi )=e^{-1/\xi ^{2}}$%
, $\xi \geq 0$, and $\psi _{1}(\xi )=\psi (\xi )\psi (1-\xi )$. Then $\psi
_{1}\in C^{\infty }(\mathbb{R}) $ and $\operatorname{supp}\psi_1=[0,1]$. Set also $\phi (\xi )=\psi
_{1}(\xi +\frac{5}{4})$, and
\begin{equation*}
\phi _{j}(\xi )=\phi (2^{j}(\xi -1)),\text{ }j\in \mathbb{N}.
\end{equation*}%
Then $\phi _{j}(\xi )$ is a $C^{\infty }$ function with support in $%
I_{j}=[1-5/2^{j+2},1-1/2^{j+2}]$. The functions
\begin{equation*}
\chi _{j}(\xi )=\frac{\phi _{j}(\xi )}{\sum_{i\geq 0}\phi _{i}(\xi )},\text{
}
\end{equation*}%
form the required partition of unity.

Set
\begin{equation*}
\chi _{j,r}(\xi )=\chi _{j}\left( (\xi /r)^{2}\right) ,\text{ }
\end{equation*}%
and%
\begin{equation*}
h_{j,r}^z(\xi ):=h_{R}^{z}(\xi )\chi _{j,r}(\xi ).
\end{equation*}

Consider the operator
\begin{equation}
T_{j,r}^z:=s_{j,r}^z(\sqrt{-\Delta-\|\rho\|^2})=h_{j,r}^z(\sqrt{-\Delta-\|\rho\|^2 })e^{-1/r^{2}(-\Delta )}.
\label{Tj,r} 
\end{equation}
Note that by (\ref{Tj,r}) and (\ref{heatexpr}),
\begin{align}
\sum_{j\in \mathbb{N}}T_{j,r}^z&=\sum_{j\in \mathbb{N}}h_{r,j}^z(\sqrt{-\Delta-\|\rho\|^2 })e^{-1/r^{2}(-\Delta
)}\notag \\&=h_{r}^z(\sqrt{-\Delta-\|\rho\|^2 })e^{-1/r^{2}(-\Delta
)}=s_{R}^{z}(\sqrt{-\Delta-\|\rho\|^2} ).  \label{sum}
\end{align}
Denote by $\kappa _{j,r}^z$ the kernel of the operator $T_{j,r}^z$. Then, (\ref%
{Tj,r}) implies that
\begin{align}
\kappa _{j,r}^z(x)&=T_{j,r}^z\delta _{x_0}\left( x\right) =h_{j,r}^z(\sqrt{-\Delta-\|\rho\|^2}%
)e^{-1/r^{2}(-\Delta )}\delta _{x_0}\left( x\right) \notag \\ &=h_{j,r}^z(\sqrt{-\Delta-\|\rho\|^2}
)p_{1/r^{2}}(x),  \label{kj,r}
\end{align}
where $x_0$ is the basepoint on $X$.
Consequently, (\ref{sum}) and (\ref{kj,r}) imply that
\begin{equation}
\kappa _{R}^{z}=\sum\limits_{j\in \mathbb{N}}\kappa^z _{j,r}.  \label{sumk}
\end{equation}%
So, to estimate the kernel $\kappa _{R}^{z}$, it suffices to estimate the
kernels $\kappa _{j,r}^z$, which by (\ref{kj,r}) are expressed in terms of the
heat kernel $p_{t}$ of $X$ and the functions $h_{j,r}^z$. For that, we shall
first recall from \cite[p.214]{ALEXOLO} some properties of
the functions $h_{j,r}^z$ we shall use in the sequel.

There is a $c>0$ such that
\begin{equation}
|\operatorname{supp}h_{j,r}^z|\leq cr2^{-j},  \label{alexo6}
\end{equation}%
\cite[p.214]{ALEXOLO}. Note that the functions $\chi _{j}$,
as well as $h_{j,r}^z$ are radial and thus invariant by the Weyl group \cite[p.612]{AN}.

Note also that for every $k\in \mathbb{N}$, there is a $c_{k}>0$, such that for every $r>0$, it holds
\begin{equation}
\Vert\chi_{j,r}^{(k)}\Vert _{\infty }\leq c_{k}r^{-k}2^{kj},\quad \Vert
{h_{j,r}^z}^{(k)}\Vert _{\infty }\leq c_{k}r^{-k}2^{-(\operatorname{Re}z-k)j}.
\label{alexo7}
\end{equation}%
As it is mentioned in \cite[p.214]{ALEXOLO}, the estimates (\ref{alexo6}) and (\ref{alexo7}) imply that for every $k\in \mathbb{N}$,
there is a $c_{k}>0$ such that
\begin{equation}
\int_{|t|\geq s}|\hat{h}_{j,r}^z(t)|dt\leq c_{k}s^{-k}r^{-k}2^{(k-\operatorname{Re}%
z)j},\;s>0,  \label{hhat}
\end{equation}%
where $\hat{h}_{j,r}^z$ is the euclidean Fourier transform of $h_{j,r}^z$.

\begin{lemma}
\label{L1ball} Let $\kappa _{R}^{z}$ be the kernel of the Riesz mean
operator $S_{R}^{z}$. Then, there is $c>0$, independent of $R$, such that for $%
\operatorname{Re}z>n/2$,
\begin{equation*}
\Vert \kappa _{R}^{z}\Vert _{L^{1}(B(0,1))}\leq c.
\end{equation*}
\end{lemma}

\begin{proof}
For the proof we shall consider different cases. Recall that $R\geq \Vert
\rho \Vert ^{2}$\textit{.}

\textit{Case 1: }$\Vert \rho \Vert ^{2}\leq R\leq \Vert \rho \Vert ^{2}+1$.

Combining (\ref{ostellari}) and the heat semigroup property, we get that
\begin{align}
\Vert p_{t}\Vert _{L^{2}(X)}& =\left( \int_{X}p_{t}(x,y)p_{t}(y,x)dy\right)
^{1/2}  \label{L2heat} \\
& \leq p_{2t}(x,x)^{1/2}\leq ct^{-n/4}.  \notag
\end{align}%
Thus, using (\ref{alexo7}), (\ref{kj,r}), (\ref{L2heat}) and  (\ref{vol}) we have
\begin{align}
\Vert \kappa^z_{j,r}\Vert _{L^{1}(B(0,1))}& \leq |B(0,1)|^{1/2}\Vert \kappa
_{j,r}^z\Vert _{L^{2}(X)}  \notag \\
& \leq c\Vert h_{j,r}^z(\sqrt{-\Delta-\|\rho\|^2})\Vert _{L^{2}\rightarrow L^{2}}\Vert
p_{1/r^{2}}\Vert _{L^{2}(X)}   \notag \\
& \leq c\Vert h_{j,r}^z\Vert _{\infty }(1/r^{2})^{-n/4}  \label{case1} \\
& \leq c(n,\Vert \rho \Vert )2^{-j\operatorname{Re}z}.  \notag
\end{align}%
So,
\begin{equation*}
\Vert \kappa _{R}^{z}\Vert _{L^{1}(B(0,1))}\leq \sum_{j\in \mathbb{N}}\Vert
\kappa^z_{j,r}\Vert _{L^{1}(B(0,1))}\leq c\sum_{j\in \mathbb{N}}2^{-j\operatorname{Re}%
z}\leq c,
\end{equation*}
since $\operatorname{Re}z>0$.

\textit{Case 2: }$R\geq \Vert \rho \Vert ^{2}+1.$

Recall that $r=\sqrt{R}$. So, the ball $B(0,1/r)$ is contained in the unit
ball. Next, let $i\geq 0$ be such that $2^{i-1}<r\leq 2^{i}$ and consider
the annulus $A_{p}=\{x\in X:2^{p}\leq |x|\leq 2^{p+1}\}$, with $p \geq -i$.
We write
\begin{equation*}
B(0,1)\subseteq B(0,1/r)\bigcup_{p=-i}^{0}A_{p}.
\end{equation*}

Applying (\ref{alexo7}), (\ref{kj,r}), (\ref{L2heat}) and (\ref{vol}) and  proceeding as
in Case 1, we have
\begin{align*}
	\Vert \kappa^z_{j,r}\Vert _{L^{1}(B(0,1/r))}& \leq |B(0,1/r)|^{1/2}\Vert \kappa^z
	_{j,r}\Vert_{L^{2}(X)}  \notag \\
	& \leq c_nr^{-n/2}\Vert h_{j,r}^z(\sqrt{-\Delta-\|\rho\|^2})\Vert _{L^{2}\rightarrow L^{2}}\Vert
	p_{1/r^{2}}\Vert _{L^{2}(X)}  \notag \\
	& \leq c_n r^{-n/2}\Vert h_{j,r}^z\Vert _{\infty }(1/r^{2})^{-n/4}  \\
	&=c_n\Vert h_{j,r}^z\Vert _{\infty }\\
	& \leq c_n2^{-j\operatorname{Re}z},  \notag
	\end{align*}
that is
\begin{equation}
\Vert \kappa _{j,r}^z\Vert _{L^{1}(B(0,1/r))}\leq c2^{-j\operatorname{Re}z}.
\label{B1/r}
\end{equation}

So, to finish the proof of the lemma it remains to prove estimates of the
kernels $\kappa _{j,r}^z$ on the annulus $A_{p}$. For that, we shall use the
fact that the kernel $G_{t}(x,y)$, $x,y\in X$, of the wave operator $\cos (t\sqrt{-\Delta-\|\rho\|^2})$, propagates with finite speed \cite[p.19]{TAYLOR2}, that is
\begin{equation}
\operatorname{supp}(G_{t})\subset \{(x,y):\;d(x,y)\leq |t|\}.  \label{finitespeed}
\end{equation}
As observed by the authors, \cite[pp.39-40]{TAYLOR2}, we may use the following formula for even functions $f(\lambda)$:
\begin{equation}\label{fourier}f(\sqrt{-\Delta-\|\rho\|^2})=\frac{1}{2\pi}\int_{-\infty}^{+\infty}\hat{f}(t)\cos( t\sqrt{-\Delta-\|\rho\|^2})dt.
\end{equation} 
Since $h_{j,r}^z$ is even,  by (\ref{fourier}) we have
{\begin{align}
\kappa _{j,r}^z(x)& =[h_{j,r}^z(\sqrt{-\Delta-\|\rho\|^2})p_{r^{-2}}(\cdot )](x)  
\label{integk} \\
& =(2\pi )^{-1/2}\int_{-\infty }^{+\infty }\hat{h}_{j,r}^z(t)[\cos t (\sqrt{-\Delta-\|\rho\|^2})p_{r^{-2}}(\cdot )]( x)dt.  \notag
\end{align}%
So, if $x\in A_{p}$, then
\begin{align}
\kappa _{j,r}^z( x)& =(2\pi )^{-1/2}\int_{-\infty }^{+\infty }\hat{h}%
_{j,r}^z(t)[\cos (t\sqrt{-\Delta-\|\rho\|^2})p_{r^{-2}}(\cdot )\mathbf{1}_{\{| y|\leq
2^{p-1}\}}]( x)dt  \notag \\
& +(2\pi )^{-1/2}\int_{-\infty }^{+\infty }\hat{h}_{j,r}^z(t)[\cos (t\sqrt{-\Delta-\|\rho\|^2})p_{r^{-2}}(\cdot )\mathbf{1}_{\{| y|>2^{p-1}\}}]( x)dt.  \notag \\
& =(2\pi )^{-1/2}\int_{|t|\geq 2^{p-1}}\hat{h}_{j,r}^z(t)[\cos (t\sqrt{-\Delta-\|\rho\|^2})
p_{r^{-2}}(\cdot )\mathbf{1}_{_{\{| y|\leq 2^{p-1}\}}}]( x)dt  \notag \\
& +(2\pi )^{-1/2}\int_{-\infty }^{+\infty }\hat{h}_{j,r}^z(t)[\cos (t\sqrt{-\Delta-\|\rho\|^2})p_{r^{-2}}(\cdot )\mathbf{1}_{\{| y|>2^{p-1}\}}]( x)dt, \label{kjr}
\end{align}}
where in the last equality we have used the finite propagation speed of the
wave operator:
if $|y|\leq 2^{p-1}$ and $|x|\geq 2^p$, then
(\ref{finitespeed}) implies that $|t|\geq 2^{p-1}$.

{So, using (\ref{integk}), equality (\ref{kjr}) rewrites
	\begin{align}\label{k_j,r}
	\kappa _{j,r}^z( x)& =(2\pi )^{-1/2}\int_{|t|\geq 2^{p-1}}\hat{h}_{j,r}^z(t)[\cos (t\sqrt{-\Delta-\|\rho\|^2})
	p_{r^{-2}}(\cdot )\mathbf{1}_{_{\{| y|\leq 2^{p-1}\}}}]( x)dt  \notag \\
	&+{h}_{j,r}^z(\sqrt{-\Delta-\|\rho\|^2})[p_{r^{-2}}(\cdot )\mathbf{1}_{\{| y|>2^{p-1}\}}]( x).
	\end{align} Applying Cauchy-Schwarz to (\ref{k_j,r}) and using the fact that $\Vert \cos t\sqrt{-\Delta }\Vert
_{2\rightarrow 2}\leq 1$, as well as the spectral theorem, we obtain }
\begin{align}
\Vert \kappa _{j,r}^z\Vert _{L^{1}(A_{p})}& \leq c|A_{p}|^{1/2}\int_{|t|\geq
2^{p-1}}|\hat{h}_{j,r}^z(t)|\Vert p_{r^{-2}}\Vert _{2}dt  \label{L1norm} \\
& +c|A_{p}|^{1/2}\Vert h_{j,r}^z\Vert _{\infty }\Vert p_{r^{-2}}\mathbf{1}%
_{\{|y|>2^{p-1}\}}\Vert _{2}:=I_{1}+I_{2}.  \notag
\end{align}

From (\ref{alexo7}), (\ref{grigoryan}) and the fact that $2^{i-1}<r\leq
2^{i} $, it follows that
\begin{align*}
I_{2}& \leq c2^{p/2}2^{-j\operatorname{Re}z}(r^{-2})^{-n/4}e^{-2^{p-1}/2Dr^{-2}} \\
& \leq c2^{-j\operatorname{Re}z}2^{p/2}r^{n/2}e^{-2^{p}r^{2}/4D} \\
& \leq c2^{-j\operatorname{Re}z}2^{(p+i)n/2}e^{-D_1 2^{p+i}}.
\end{align*}%
Using the elementary estimate
\begin{equation*}
e^{-D_1x}x^{n/2}\leq c_kx^{-k},\text{ for all }x>1,\;k\in \mathbb{N},
\end{equation*}%
we obtain
\begin{equation}
I_{2}\leq 2^{-j\operatorname{Re}z}2^{-k(p+i)}.  \label{I1}
\end{equation}%
Also, from (\ref{L2heat}) we have that
\begin{equation*}
I_{1}\leq c2^{p/2}(r^{-2})^{-n/4}\int_{|t|\geq 2^{p-1}}|\hat{h}_{j,r}^z(t)|dt.
\end{equation*}%
Then, applying (\ref{hhat}) for $k>n/2$, we obtain
\begin{align}
I_{1}& \leq c_{n}2^{(p+i)n/2}2^{-pk}r^{-k}2^{(k-\operatorname{Re}z)j}  \notag
\label{I2} \\
& \leq c2^{-(p+i)(k-n/2)}2^{-j(\operatorname{Re}z-n/2)}.
\end{align}%
Finally, using (\ref{I1}) and (\ref{I2}), (\ref{L1norm}) implies that
\begin{equation}
\Vert \kappa _{j,r}^z\Vert _{L^{1}(A_{p})}\leq c2^{-(p+i)(k-n/2)}2^{-j(\operatorname{Re%
}z-n/2)}.  \label{L1Ap}
\end{equation}%
\textit{End of proof of Lemma \ref{L1ball}.} It follows from (\ref{B1/r})
and (\ref{L1Ap}) that
\begin{align}
\Vert \kappa _{j,r}^z\Vert _{L^{1}(B(0,1))}& \leq c2^{-j\operatorname{Re}%
z}+c\sum_{p=-i}^{0}2^{-(p+i)(k-n/2)}2^{-j(\operatorname{Re}z-n/2)} \\
& \leq c2^{-j(\operatorname{Re}z-n/2)}.  \notag
\end{align}%
So, for $\operatorname{Re}z>n/2$,
\begin{align*}
\Vert \kappa _{R}^{z}\Vert _{L^{1}(B(0,1))} &\leq c\sum_{j\geq 0}\Vert
\kappa _{j,r}^z\Vert _{L^{1}(B(0,1))} \\
&\leq c\sum_{j\geq 0}2^{-j(\operatorname{Re}z-n/2)}\leq c.
\end{align*}
\end{proof}

\begin{lemma}
$S_{R}^{z,0}$ is bounded on $L^{2}(X)$.
\end{lemma}

\begin{proof}
Set
\begin{equation}
\kappa _{j,r}^{z,0}=\zeta \kappa _{j,r}^z,\;T_{j,r}^{z,0}=\ast \kappa _{j}^{z,0}%
\text{ and }s_{j,r}^{z,0}=\mathcal{H}(\kappa _{j,r}^{z,0}),  \label{cutoff}
\end{equation}%
where $\zeta $ is the cut-off function given in (\ref{zita}).

By Plancherel theorem and using (\ref{cutoff}), we get that
\begin{align}
\Vert T_{j,r}^{z,0}\Vert _{L^{2}\rightarrow L^{2}}& \leq \Vert
s_{j,r}^{z,0}\Vert _{L^{\infty }(\mathfrak{a^{\ast }})}=\Vert \mathcal{H}%
(\kappa _{j,r}^{z,0})\Vert _{L^{\infty }(\mathfrak{a^{\ast }})}  \notag \\
& =\Vert \mathcal{H}(\zeta \kappa _{j,r}^z)\Vert _{L^{\infty }(\mathfrak{%
a^{\ast }})}=\Vert \mathcal{H}(\zeta )\ast \mathcal{H}(\kappa _{j,r}^z)\Vert
_{L^{\infty }(\mathfrak{a^{\ast }})}  \label{tj02} \\
& \leq \Vert \mathcal{H}(\zeta )\Vert _{L^{1}(\mathfrak{a^{\ast }})}\Vert
s_{j,r}^z\Vert _{L^{\infty }(\mathfrak{a^{\ast }})}.  \notag
\end{align}%
But $\zeta \in S(K\backslash G/K)$. Therefore, its spherical Fourier
transform $\mathcal{H}(\zeta )$, belongs in $S(\mathfrak{a^{\ast }}%
)^{W}\subset L^{1}(\mathfrak{a^{\ast }})$, (see Section 2). So,
\begin{equation*}
\Vert \mathcal{H}(\zeta )\Vert _{L^{1}(\mathfrak{a^{\ast }})}\leq c(\zeta
)<\infty .
\end{equation*}%
From (\ref{tj02}), (\ref{Tj,r}) and (\ref{alexo7}) it follows that
\begin{align}
\Vert T_{j,r}^{z,0}\Vert_{L^{2}\rightarrow L^{2}}&\leq c(\zeta)\Vert
s_{j,r}^z\Vert _{L^{\infty }(\mathfrak{a^{\ast }})}\leq c(\zeta) \| h_{j,r}^z(\sqrt{\cdot })e^{-1/r^{2}(\cdot)}\|_{L^{\infty }(\mathfrak{a^{\ast }})} \notag \\&\leq c(\zeta
) \| h_{j,r}^z(\sqrt{\cdot })\|_{L^{\infty }(\mathfrak{a^{\ast }})}\leq
c(\zeta )2^{-j\operatorname{Re}z}.  \label{Tjr0L2}
\end{align} Further, by (\ref{Tjr0L2}) and the fact that $%
S_{R}^{z,0}=\sum_{j\geq 0}T_{j,r}^{z,0}$, it follows that
\begin{equation}
\Vert S_{R}^{z,0}\Vert _{L^{2}\rightarrow L^{2}} \leq \sum_{j\geq 0}\Vert
T_{j,r}^{z,0}\Vert _{L^{2}\rightarrow L^{2}} \leq c\sum_{j\geq 0}2^{-j\operatorname{Re}%
z}\leq c<\infty .  \label{t2}
\end{equation}
\end{proof}

\textit{End of the proof of Proposition \ref{T0X}}: Since $\kappa
_{R}^{z}=\sum\limits_{j\geq 0}\kappa _{j,r}^z$, by Lemma \ref{L1ball}, we have
\begin{equation*}
\Vert \kappa _{R}^{z,0}\Vert _{L^{1}(X)}=\Vert \zeta \kappa _{R}^{z}\Vert
_{L^{1}(X)}\leq c\Vert \kappa _{R}^{z}\Vert _{L^{1}(B(0,1))}<c.
\end{equation*}%
This implies that
\begin{equation}
\Vert S_{R}^{z,0}\Vert _{L^{\infty }\rightarrow L^{\infty }}\leq c(z).
\label{tapeiro}
\end{equation}

By interpolation and duality, it follows from (\ref{tapeiro}) and (\ref{t2}),
that for all
$p\in \lbrack 1,\infty ]$, $\Vert S_{R}^{z,0}\Vert _{p\rightarrow p}\leq
c(z)$, with $\operatorname{Re}z>n/2$.  

\subsection{The part at infinity}

For the part at infinity $S_R^{z, \infty}$ of the operator, we proceed as in
\cite{LOMAanna} to obtain estimates of its kernel $\kappa_{R}^{z, \infty}$. Let $l=\operatorname{rank}(X)$.

To begin with, recall that $\kappa _{R}^{z}=\mathcal{H}^{-1}s_{R}^{z}$.
Recall also the following result from \cite[p.650]{LOMAanna}, based on the Abel transform conservation property. 
\begin{lemma} For $x=k_1(\exp H)k_2\in G$, with $|x|>1$ and $k\in \mathbb{N}$ with $k>\frac{n}{2}-\frac{l}{4}$, we have that
\begin{equation}
|\kappa _{R}^{z}(x)|\leq c\varphi _{0}(x)\left( \int\limits_{|H|>|x|-\frac{1%
}{2}}\left( \sum_{|\alpha |\leq 2k}|\partial _{H}^{\alpha }(\mathcal{F}%
^{-1}s_{R}^{z})(H)|\right) ^{2}\right) ^{1/2}.  \label{kernelLOMA}
\end{equation}%
\end{lemma}
Thus, to estimate the kernel for $|x|>1$, it suffices to obtain estimates
for the derivatives of the euclidean inverse Fourier transform of $%
s_{R}^{z}(\lambda )$. Denote by $\mathcal{J}_{\nu }(t)=t^{-\nu }J_{\nu }(t)$, $t>0$, where $J_{\nu }$
	is the Bessel function of order $\nu $. Then, it holds
\begin{equation}
(\mathcal{F}^{-1}s_{R}^{z})(\exp H)=c(n,z)R^{-z}(R-\Vert \rho \Vert
^{2})^{z+l/2}\mathcal{J}_{z+l/2}\left( \sqrt{R-\Vert \rho \Vert ^{2}}%
|H|\right) ,  \label{Fourier}
\end{equation}
\cite{DAVCHANG, GIUMA}, and we shall need the following auxiliary lemma.

\begin{lemma}
For every multi-index $\alpha $%
, it holds that
\begin{equation}
|\partial _{H}^{\alpha }\mathcal{J}_{z+l/2}(\sqrt{R-\Vert \rho \Vert ^{2}}%
|H|)|\leq c(R-\Vert \rho \Vert ^{2})^{\frac{|\alpha |}{2}-(\frac{\operatorname{Re}z}{%
2}+\frac{l+1}{4})}|H|^{-(\operatorname{Re}z+\frac{l+1}{2})}.  \label{Bessel deriv}
\end{equation}
\end{lemma}

\begin{proof}
Using the identity $\mathcal{J}_{\nu }^{\prime }(t)=-t\mathcal{J}_{\nu
+1}(t) $, it is straightforward to get that
\begin{equation}
\mathcal{J}_{\nu }^{(a)}(t)=(-1)^{a}t^{a}\mathcal{J}_{\nu
+a}(t)+\sum_{j=1}^{[a/2]}c_{j}^{a}t^{a-2j}\mathcal{J}_{\nu +a-j}(t),\;a\in
\mathbb{N},  \label{nth derivative}
\end{equation}%
for some constants $c_{j}^{a}$, where $[a]$ denotes the integer part of $a$.
Applying the inequality
\begin{equation*}
\ |\mathcal{J}_{\mu }(t)|\leq c_{\mu }t^{-(\operatorname{Re}\mu +1/2)},\text{ for
all }t>0,
\end{equation*}%
\cite{GIUMA}, it follows that
\begin{equation*}
|\partial _{H}^{\alpha }\mathcal{J}_{\nu }(\sqrt{R-\Vert \rho \Vert ^{2}}%
|H|)|\leq c(R-\Vert \rho \Vert ^{2})^{\frac{|\alpha |}{2}-(\frac{\operatorname{Re}%
\nu }{2}+\frac{1}{4})}|H|^{-(\operatorname{Re}\nu +\frac{1}{2})}
\end{equation*}%
and (\ref{Bessel deriv}) follows by taking $\nu =z+l/2$.
\end{proof}

\begin{lemma}
\label{infinity} If $R\geq \|\rho\|^2+1$, then
\begin{equation}  \label{kappainf}
|\kappa_R^z(x)|\leq c\varphi_{0 }(x)R^{-\frac{1}{2}(\operatorname{Re}z-n+\frac{1}{2}%
)}|x|^{-\operatorname{Re}z-\frac{1}{2}}, \quad |x|>1.
\end{equation}
\end{lemma}

\begin{proof}
From (\ref{Bessel deriv}), we get that
\begin{align}
I^{2}& :=\int\limits_{|H|>|x|-\frac{1}{2}}\left( \sum_{|\alpha |\leq
2k}\left\vert \partial _{H}^{a}\mathcal{J}_{z+l/2}\left( \sqrt{R-\Vert \rho
\Vert ^{2}}|H|\right) \right\vert \right) ^{2}dH  \notag  \label{estimate} \\
& \leq c\left( \sum_{|\alpha |\leq 2k}(R-\Vert \rho \Vert ^{2})^{a/2}\right)
^{2}\times\notag \\ &\times \int\limits_{|H|>|x|-\frac{1}{2}}\left( (R-\Vert \rho \Vert ^{2})^{-(%
\frac{\operatorname{Re}z}{2}+\frac{l+1}{4})}|H|^{-(\operatorname{Re}z+\frac{l+1}{2})}\right)
^{2}dH  \notag \\
& \leq c(R-\Vert \rho \Vert ^{2})^{-2(\frac{\operatorname{Re}z}{2}+\frac{l+1}{4}%
)+2k}\int\limits_{u>|x|-\frac{1}{2}}u^{-(l+1)-2\operatorname{Re}z}u^{l-1}du  \notag
\\
& \leq c(R-\Vert \rho \Vert ^{2})^{-2(\frac{\operatorname{Re}z}{2}+\frac{l+1}{4}%
)+2k}\left( |x|-\frac{1}{2}\right) ^{-2\operatorname{Re}z-1}.
\end{align}
For $R\geq \Vert \rho \Vert ^{2}+1$, since $k>\frac{n}{2}-\frac{l}{4}$, we
have that
\begin{equation}
I\leq c(R-\Vert \rho \Vert ^{2})^{-(\frac{\operatorname{Re}z+l-n}{2}+\frac{1}{4}%
)}\left( |x|-\frac{1}{2}\right) ^{-\operatorname{Re}z-\frac{1}{2}}.
\label{estimate1}
\end{equation}%
Using (\ref{estimate1}) and (\ref{Fourier}), from (\ref{kernelLOMA}) we
obtain that
\begin{align*}
|\kappa _{R}^{z}(x)|& \leq c\varphi _{0}(x)R^{-\operatorname{Re}z}(R-\Vert \rho
\Vert ^{2})^{\operatorname{Re}z+\frac{l}{2}}\times \notag \\ &\times (R-\Vert \rho \Vert ^{2})^{-(\frac{\operatorname{%
			Re}z+l-n}{2}+\frac{1}{4})}\left( |x|-\frac{1}{2}\right) ^{-\operatorname{Re}z-\frac{1%
}{2}} \\
& \leq c\varphi _{0}(x)R^{-\frac{1}{2}(\operatorname{Re}z-n+\frac{1}{2})}|x|^{-\operatorname{%
Re}z-\frac{1}{2}},\quad |x|>1.
\end{align*}
\end{proof}

Using the estimate (\ref{estimate1}) and proceeding as above, one can prove
the following result.

\begin{lemma}
\label{extracase}If $\Vert \rho \Vert ^{2}\leq R\leq \Vert \rho \Vert
^{2}+1, $ then
\begin{equation*}
|\kappa _{R}^{z}(x)|\leq c\varphi _{0}(x)|x|^{-\operatorname{Re}z-\frac{1}{2}%
},\;|x|>1.
\end{equation*}
\end{lemma}
Finally, we shall prove the following result, which, combined with Proposition \ref{T0X}, finishes the proof of Theorem 1.
\begin{proposition}\label{11}
\label{TinfX} Let $\operatorname{Re}z\geq n-\frac{1}{2}$ and consider $q>2$. Then
for every $p$ such that $1\leq p\leq q^{\prime }$, $S_{R}^{z,\infty }$ is
continuous from $L^{p}(X)$ to $L^{r}(X)$ for every $r\in \lbrack qp^{\prime
}/(p^{\prime }-q),\infty ]$, and $\Vert S_{R}^{z,\infty }\Vert
_{p\rightarrow r}\leq c(z)$ for all $R\geq \left\Vert \rho \right\Vert^2 $.
\end{proposition}

\begin{proof}
Recall that $\kappa_{R}^{z, \infty}(x)=\kappa_{R}^z(x)$ for every $|x|>1$. Using the estimates of $\kappa _{R}^{z }$ from Lemmata \ref{infinity}
and \ref{extracase}, as well as the estimate (\ref{phi0}), it follows that $%
\kappa _{R}^{z,\infty }$ is in $L^{q}(X)$ for every $q>2$. Thus, by Young's
inequality, the operator $f\rightarrow |f|\ast \kappa _{R}^{z,\infty }$ maps
$L^{p}(X)$, $p\in \lbrack 1,q^{\prime }]$, continuously into $L^{r}(X)$, for
every $r\in \lbrack qp^{\prime }/(p^{\prime }-q),\infty ]$.

Further, for $z\geq n-\frac{1}{2}$, in Lemmata \ref{infinity} and \ref%
{extracase} the estimates of the kernel $\kappa _{R}^{z,\infty }$ are uniform with respect to $R$. This implies that the norm $\Vert S_{R}^{z,\infty }\Vert
_{p\rightarrow r}$ is bounded by a constant, uniform with respect to $R$.
\end{proof}

\section{Proof of Theorem 2 and Theorem 3}

In this section we give the proof of Theorem \ref{Th1}, which deals with the
$L^{p}$-continuity of the maximal operator $S_{\ast }^{z}$ associated with
the Riesz means. This allows us to deduce the almost everywhere convergence
of Riesz means $S_{R}^{z}(f)$ to $f$, as $R\rightarrow +\infty $.

Recall first that
\begin{equation}
S_{\ast }^{z}(f)=\sup_{R>\Vert \rho \Vert ^{2}}|S_{R}^{z}(f)|,\;f\in L^{p}({X%
}).  \label{maximal}
\end{equation}
The following proposition holds true, \cite[Lemma 4.1]{GIUMA}.
\begin{proposition}\label{S*L2}
	Let $\operatorname{Re}z>0$. Then, $S_*^{z}$ is continuous on $L^2(X)$.
\end{proposition}

Recall the following decomposition of the kernel $\kappa_{R}^z$ of the operator $S_R^z$:
\begin{equation}  \label{decomp1}
\kappa_{R}^{z}=\zeta \kappa_{R}^{z}+(1-\zeta
)\kappa_{R}^z:=\kappa_{R}^{z,0}+\kappa_{R}^{z,\infty },
\end{equation}%
where $\zeta \in C^{\infty }(K\backslash G/K)$ is a cut-off function such
that 
\begin{equation}
\zeta (x)= 
\begin{cases}
1, & \text{if }|x|\leq 1/2, \\ 
0, & \text{if }|x|\geq 1.%
\end{cases}
\label{zita1}
\end{equation}
Denote by ${S}_R^{z,0}$ (resp. ${S}_R^{z,\infty}$) the convolution operators
on $X$ with kernel $\kappa _R^{z,0}$ (resp. $\kappa_R^{z,\infty})$. Then, 
\[\
S_*^zf\leq\sup\limits_{R\geq \|\rho\|^2}|S_R^{z,0}f|+\sup\limits_{R\geq \|\rho\|^2}|S_R^{z, \infty}f|.
\]
The following holds true for the part at infinity $S_{\ast}^{z, \infty}$ of the operator $S_{\ast}^{z}$.
\begin{proposition}\label{propinf}
	\label{TinfXR} Let $ \operatorname{Re}z \geq n-\frac{1}{2}$. Then, for every $q>2$ and $p\in[1, q^{\prime}]$, $S_*^{z,\infty}$ is continuous from $L^p(X)$ to
	$L^r(X)$ for every $r\in[qp^{\prime}/(p^{\prime} - q), \infty]$.
\end{proposition}

	The proof relies on the uniform kernel estimates for $\kappa_{R}^{z, \infty}$ implied by Lemmata 9 and 10. It is similar to the proof of Proposition 11, thus omitted.

We shall now prove the following result concerning the local part $S_{\ast}^{z, 0}$ of the Riesz means maximal operator.
\begin{proposition}\label{proplocal}
	Let $\operatorname{Re}z\geq n-\frac{1}{2}$. Then, $S_*^{z,0}$ is continuous on $L^p(X)$, for every $p \in (1, \infty)$, and it maps $L^1(X)$ continuously into $L^{1,w}(X)$.
\end{proposition}

Denote by $e^{t\Delta}, \;t>0$, the heat operator on $X$. Then, $e^{t\Delta}=\ast p_t$, where $p_t$ is the heat kernel on $X$. Recall that $p_t$ is given as the inverse spherical Fourier transform of 
\begin{equation*}
w_t(\lambda)=e^{-t(\|\lambda\|^2+\|\rho\|^2)}, \; \lambda\in \mathfrak{\ a^{\ast}}.
\end{equation*}Consider the radial multiplier
\begin{equation}\label{Mdef1}
M(R^{-1}\lambda):=s_R^z(\lambda)-w_{R^{-1}}(\lambda), \; R\geq \|\rho\|^2.
\end{equation}

Denote by $K_R(x)$ the kernel of the operator $M(-R^{-1}\Delta)$ and set $K_R^0(x):=\zeta(x)K_R(x)$. Similarly, set $s_R^{z,0}=\mathcal{H}(\zeta \kappa_{R}^{z})=\mathcal{H}( \kappa_{R}^{z,0})$ and  $w_{R^{-1}}^0=\mathcal{H}(\zeta p_{R^{-1}})=\mathcal{H}( p^0_{R^{-1}})$. Then, using (\ref{Mdef1}), we have that
\begin{equation}\label{mult1}
\mathcal{H}(\kappa_{R}^0):=M^0(-R^{-1}\cdot)=s_R^{z,0}-w_{R^{-1}}^0,
\end{equation}
From (\ref{mult1}) we have that
\begin{equation} \label{1,0}
S_*^{z,0}f=\sup_{R\geq\|\rho\|^2}|s_R^{z,0}(-\Delta)f|\leq \sup_{R\geq\|\rho\|^2}|M^0(-R^{-1} \Delta)f|+\sup_{R\geq\|\rho\|^2}|f \ast p^0_{R^{-1}} |.
\end{equation}

Consider the operator $(-\Delta)^{i\gamma}$, $\gamma \in \mathbb{R}$, which in the spherical Fourier transform variables is given by
\[\
\mathcal{H}((-\Delta)^{i\gamma}f)=(\|\lambda\|^2+\|\rho\|^2)^{i\gamma}\mathcal{H}(f), \; \lambda\in \mathfrak{\ a^{\ast}}.\]
Denote by 
\[
\kappa^{\gamma}=\mathcal{H}^{-1}((\|\lambda\|^2+|\rho\|^2)^{i\gamma})
\]
the kernel of $(-\Delta)^{i\gamma}$.  As in \cite{ALEXOLO, GIUMA}, using the Mellin transform $\mathcal{M}(\gamma)$ of the radial function $M(\lambda)$, one can express the operator $M(-R^{-1}\Delta)$ as follows:
\begin{equation}\label{Mellin1}
M(-R^{-1}\Delta)=\int_{-\infty}^{+\infty}\mathcal{M}(\gamma)R^{-i\gamma}(-\Delta)^{i\gamma}d\gamma, 
\end{equation}
where
\begin{equation}\label{calM}
|\mathcal{M}(\gamma)|\leq c(1+|\gamma|)^{-(\operatorname{Re}z+1)},
\end{equation}
\cite{GIUMA}. Using (\ref{Mellin1}), the kernel $K_R$ of $M(-R^{-1}\Delta)$ is given by
\begin{equation*}
K_R=\int_{-\infty}^{+\infty}\mathcal{M}(\gamma)R^{-i\gamma} \kappa^{\gamma}d\gamma,
\end{equation*} and thus
	\begin{align*}
	K_R^0(x)&=\zeta(x) K_R(x)=\int_{-\infty}^{+\infty}\mathcal{M}(\gamma)R^{-i\gamma} \zeta(x)\kappa^{\gamma}(x)d\gamma\\
	&=\int_{-\infty}^{+\infty}\mathcal{M}(\gamma)R^{-i\gamma} \kappa^{\gamma,0}(x)d\gamma .
	\end{align*} 
It follows that 
	\[\
	M^0(-R^{-1}\Delta)=\int_{-\infty}^{+\infty}\mathcal{M}(\gamma)R^{-i\gamma} (-\Delta)^{i\gamma,0}d\gamma .
	\] Hence, 
\begin{equation}\label{Mellin}
\sup_{R>\|\rho\|^2} |M^0(-R^{-1}\Delta)f|\leq \int_{-\infty}^{+\infty}|\mathcal{M}(\gamma)||(-\Delta)^{i\gamma,0}f|d\gamma
\end{equation}
\begin{lemma}\label{Anker}
	The operator $(-\Delta)^{i\gamma,0}$ is bounded on $L^p$, $p\in(1, \infty)$, with
	\begin{equation}\label{Lpnorm}
	\|(-\Delta)^{i\gamma,0}\|_{L^p\rightarrow L^p}\leq c_p(1+|\gamma|)^{[n/2]+1}.
	\end{equation}
	Moreover, the operator $(-\Delta)^{i\gamma,0}$ is also $L^1\rightarrow L^{1,w}$ bounded, with 
	\begin{equation}\label{L1wnorm}
	\|(-\Delta)^{i\gamma,0}\|_{L^1\rightarrow L^{1,w}}\leq c(1+|\gamma|)^{[n/2]+1}.
	\end{equation}
\end{lemma}
\begin{proof}
	To prove the lemma, we shall proceed as in \cite{AN}.  More precisely, by using a smooth, radial partition of unity (and thus invariant by the Weyl group), we decompose the multiplier $m^{\gamma}(\lambda)=(\|\lambda\|^2+\|\rho\|^2)^{i\gamma}$ as follows
	\[
	m^{\gamma}(\lambda)=\sum_{k=0}^{+\infty}m^{\gamma}_k(2^{-k}\lambda),
	\]
	where $\operatorname{supp}m^{\gamma}_0 \subset \{\|\lambda\|\leq 2 \}$ and $\operatorname{supp}m^{\gamma}_k \subset \{1/2\leq \|\lambda\|\leq 2 \}$ for $k\geq 1$.
	Then, for every $p\in (1, +\infty)$, we have
	\begin{equation}\label{Lpimag}
	\|(-\Delta)^{i\gamma,0}\|_{p\rightarrow p}\leq c_p \sup_{k\geq 0}\|m^{\gamma}_k\|_{H_2^{\sigma/2}}, 
	\end{equation}
	with $\sigma >n$ and $H_2^{\sigma/2}$ the usual Sobolev space, \cite[Corollary 17, ii]{AN}. Note that the same upper bound also holds for the $L^1\rightarrow L^{1, w}$ norm of $(-\Delta)^{i\gamma,0}$, \cite{AN}. A straightforward computation yields
	\begin{equation}\label{mknorm}
	\|m^{\gamma}_k\|_{H_2^{\sigma/2}}\leq c(1+|\gamma|)^{\sigma /2},
	\end{equation}
	for $\sigma /2$ an integer, and Lemma \ref{Anker} follows from (\ref{Lpimag}).
\end{proof}
\begin{proof}[End of the proof of Proposition \ref{proplocal}] We shall complete the proof for the $L^p$ boundedness of $S_{*}^{z, 0}$, $p\in(1, \infty)$; the $L^1\rightarrow L^{1,w}$ result is similar, thus omitted. Recall that (\ref{1,0}) states that
	\[
	S_{\ast}^{z,0}f\leq \sup_{R\geq \|\rho\|^2} |M^0(-R^{-1}\Delta)f|+\sup_{R\geq \|\rho\|^2}|f\ast p^0_{R^{-1}}|.
	\] Note that since $p_t(x)\geq 0$, for every $x\in X$, we have $p_t^0(x)\leq p_t(x)$. Thus, 
	\begin{equation}\label{L1w}
	|(f\ast p_t^0)(x)|\leq (|f|\ast p_t)(x).
	\end{equation}
	Also, it is known (see for example \cite[Corollary 3.2]{AN2}) that the heat maximal operator $\sup\limits_{t>0}|e^{t\Delta}f|$ is $L^p$-bounded and also $L^1\rightarrow L^{1,w}$ bounded. This implies that the operator $\sup\limits_{R\geq\|\rho\|^2}|\ast p_{R^{-1}}^0|$ is also $L^p$-bounded and $L^1\rightarrow L^{1,w}$ bounded.
	Thus, from (\ref{1,0}), it follows that to prove the $L^p$-boundedness of the operator $S_{\ast}^{z,0}$, it suffices to prove the $L^p$-boundedness of the operator $\sup\limits_{R\geq\|\rho\|^2}|M^0(-R^{-1}\Delta)|$, and similarly for the $L^1\rightarrow L^{1,w}$ boundedness. 
	
	From (\ref{Mellin}) and (\ref{mknorm}), we have that
	\begin{align*}
	\|\sup_{R\geq \|\rho\|^2} |M^0(-R^{-1}\Delta)|\|_p &\leq \int_{-\infty}^{+\infty}|\mathcal{M}(\gamma)|\|(-\Delta)^{i\gamma,0}\|_{p\rightarrow p}|\|f\|_pd\gamma\\
	&\leq c|\|f\|_p \int_{-\infty}^{+\infty}(1+|\gamma|)^{-(\operatorname{Re}z+1)})(1+|\gamma|)^{[n/2]+1}d\gamma \\
	&\leq c|\|f\|_p \int_{-\infty}^{+\infty}(1+|\gamma|)^{-(\operatorname{Re}z-[n/2])}d\gamma\leq c\|f\|_p,
	\end{align*}
	whenever $\operatorname{Re}z\geq n-\frac{1}{2}$. This completes the proof of Proposition \ref{proplocal}.
\end{proof}
\emph{Proof of Theorem 2.} The proof of Theorem 2 follows from Stein's complex interpolation, between the $L^p$ result for $p$ close to $1$ and the $L^2$ result (Propositions 12, 13 and 14).

\emph{Proof of Theorem 3.} As it is already mentioned in the Introduction, from Theorem \ref%
{Th1} and Propositions 13 and 14, and well-known measure theoretic arguments (see for example \cite[Theorem 2.1.14]{GRAF}), we deduce the almost everywhere convergence of Riesz means: if $1\leq
p\leq 2$ and $\operatorname{Re}z>\left( n-\frac{1}{2}\right) \left( \frac{2}{p}%
-1\right) $, then
\begin{equation*}
\lim\limits_{R\rightarrow +\infty }S_{R}^{z}(f)(x)=f(x),\;\text{a.e., for }%
f\in L^{p}(X).
\end{equation*}
\section*{Acknowledgement}
The authors would like to thank the anonymous referee for the valuable comments and remarks.

\end{document}